\documentclass[
11pt]{amsart}

\usepackage{times}             
\usepackage{epsf}
\usepackage{amsfonts}
\usepackage{fancyhdr}
\usepackage{amsthm}
\usepackage{amsmath}
\usepackage{graphicx}
\usepackage{url}
\usepackage{enumerate}
\usepackage{cite}
\usepackage{lscape}
\usepackage{indentfirst}
\usepackage{latexsym}
\usepackage{multirow}
\usepackage{tabls}
\usepackage{wrapfig}
\usepackage{slashbox}
\usepackage{longtable}
\usepackage{supertabular}
\usepackage{subfigure}

\begin{document}

\newtheorem{definition}{Definition}[section]
\newtheorem{theorem}[definition]{Theorem}
\newtheorem*{theorem_}{Theorem}
\newtheorem{lemma}[definition]{Lemma}
\newtheorem{question}[definition]{Question}
\newtheorem{corollary}[definition]{Corollary}
\newtheorem*{corollary_}{Corollary}
\newtheorem{proposition}[definition]{Proposition}
\newtheorem{claim}[definition]{Claim}
\newtheorem{conjecture}[definition]{Conjecture}
\newtheorem*{maintheorem}{Main Theorem}
\theoremstyle{remark}
\newtheorem{remark}{Remark}[section]
\newtheorem{notation}[remark]{Notation}
\newtheorem{example}[remark]{Example}
\newcommand{\Mob}{\textnormal{M\"ob}}
\newcommand{\Isom}{\textnormal{Isom}}
\newcommand{\Hom}{\textnormal{Hom}}
\newcommand{\Aut}{\textnormal{Aut}}
\newcommand{\Ad}{\textnormal{Ad}}

\title[Deformations of representations and earthquakes on $SO(n,1)$ surface groups]{Deformations of fundamental group representations and earthquakes on $SO(n,1)$ surface groups}
\author{Son Lam Ho\\
}
\address{Department of Mathematics\\
Universit\'{e} de Sherbrooke\\
Sherbrooke, QC Canada.
}
\thanks{}
\begin{abstract}
In this article we construct a type of deformations of representations $\pi_1(M)\rightarrow G$ where $G$ is an arbitrary lie group and $M$ is a large class of manifolds including CAT(0) manifolds. The deformations are defined based on codimension 1 hypersurfaces with certain conditions, and also on disjoint union of such hypersurfaces, i.e. multi-hypersurfaces. We show commutativity of deforming along disjoint hypersurfaces. As application, we consider Anosov surface groups in $SO(n,1)$ and show that the construction can be extended continuously to measured laminations, thus obtaining earthquake deformations on these surface groups.
\end{abstract}
\maketitle
\section{Introduction}
The Fenchel-Nielsen twist is one of the most fundamental tools in studying the deformation space of a hyperbolic surface $M$, equivalently the Teichmueller space $\mathcal{T}(M)$. Geometrically it can be described as cutting the surface along a simple closed geodesic, do a twist of some length $t$ and then glue back. However these deformations can also be described algebraically and perhaps more naturally so. That is the point of view of this paper: we will define a type of algebraic deformations of representations $\rho:\pi_1(M)\rightarrow G$ where $G$ is an arbitrary lie group and $M$ is a manifold with contractible universal cover. It turns out that many geometric deformations are special cases of this construction if we consider only the holonomy representation, for example, bending a quasifuchsian surface group defined by Thurston \cite{ThurstonNotes}, Johnson-Millson bending \cite{JohnsonMillson}, the twist-bulge deformation of a convex real projective surface by Goldman \cite{GoldmanConvexProj}, etc. Interestingly, the starting representation $\rho$ from which we deform need not be discrete. 

\subsection{Construction and results}
Let $M$ be an oriented surface of genus $g>1$ and $N\hookrightarrow M$ a directed simple closed curve, this is the case of most interest to us even though our construction applies more generally to aspherical manifolds of higher dimensions with two-sided aspherical hypersurfaces. Let $\rho:\pi_1(M)\rightarrow G$ be faithful, and suppose that the centralizer $C_G(\rho(\pi_1(N)))$ is non-trivial. For each lift $N_i$ (where $i=0,1,2,..$) of $N$ in the universal cover $\widetilde{M}$, we assign a transformation $\gamma_i$ in the $G$-centralizer of the $\rho$-image of the cyclic subgroup of $\pi_1(M)$ which preserve $N_i$ under deck transformation. This assignment of $\gamma_i$ also has to satisfy a condition of equivariance, which means that the choice of $\gamma_0$ for $N_0$ determines the rest of the $\gamma_i$'s. Fix a base point $\tilde{x}_0\in\widetilde{M}$, our construction can be informally described as follows. Given $A\in\pi_1(M)$, we consider a directed path from $\tilde{x}_0$ to $\tilde{x}_0 A$ where $\tilde{x}_0 A$ is the image of $\tilde{x}_0$ under the right-action of deck transformations. This path will cross an ordered collection of lifts $N_1,..., N_k$, we can define a new function $\mathcal{E}_{N,\gamma}(\rho):\pi_1(M)\rightarrow G$ such that:
$$
\mathcal{E}_{N,\gamma}(\rho)(A) = \rho(A) \gamma_k^{s_k} ... \gamma_1 ^{s_1} 
$$
where $s_i$ is the intersection sign between $N_i$ and the path from $\tilde{x}_0$. In fact the result does not depend on the particular path from $\tilde{x}_0$ to $\tilde{x}_0 A$, as long as this path intersects $N_i$ transversely. We can then show that the resulting map $\mathcal{E}_{N,\gamma}(\rho):\pi_1(M)\rightarrow G$ is indeed a homomorphism. Thus for example, when $G=PSL(2,\mathbb{R})$ and $\rho$ a holonomy representation of a hyperbolic surface, we can choose $\gamma_i$ inside a 1-parameter group of hyperbolic transformations and obtain a 1-parameter family of representations corresponding to the Fenchel-Nielsen twists.

In section \ref{sec:multi-hyp} we generalize the above construction to multi-curves on surfaces (and multi-hypersurfaces on manifolds). We will also prove a commutativity result, Theorem \ref{thm:commute}, which we will state below. Let $L$ be a multi-hypersurface, or a union of disjoint hypersurfaces on $M$, satisfying generic conditions. Suppose that for each component $L_i$ of lifts of $L$ we have $\gamma_i$ and $\alpha_i$ chosen equivariantly such that the deformations $\mathcal{E}_{L,\gamma}(\rho)$ and $\mathcal{E}_{L,\alpha}(\rho)$ can be defined. Then
 \begin{theorem_}(Theorem \ref{thm:commute})
  If we have $\gamma_i \alpha_i = \alpha_i \gamma_i$ then
 $$
 \mathcal{E}_{L,\gamma}(\mathcal{E}_{L,\alpha}(\rho)) =  \mathcal{E}_{L,\alpha}(\mathcal{E}_{L,\gamma}(\rho)).
 $$
 \end{theorem_}
This implies the following important corollaries: 1-parameter deformations along a hypersurface is a flow, that is, $\mathcal{E}_{L,\gamma^t}( \mathcal{E}_{L,\gamma^s}(\rho) ) = \mathcal{E}_{L,\gamma^{t+s}}(\rho)$; and deformations along disjoint hypersurfaces commute. For example, given a $PSL(2,\mathbb{C})$ surface group, bending it along a curve and twisting it along another (disjoint) curve are commutative.

The above construction and theorem are presented in section 2. A large part of this paper is in section 3 where we switch attention to the case of surface groups in $G=SO(n,1)  =\Isom(\mathbb{H}^n)$ and study the limit of deformations along weighted simple closed curves that approach a measured lamination. Let $S$ be a surface of genus $g>1$. Thurston introduced the space of measured laminations $\mathcal{ML}(S)$ which can be thought of as a completion of the space of weighted simple closed $\mathcal{C}(S)$ which is in fact dense in $\mathcal{ML}(S)$ \cite{ThurstonNotes}. The classical earthquake deformation is then defined to be the limit of Fenchel-Nielsen twists as a sequence in $\mathcal{C}(S)$ approach a measured lamination \cite{KerckhoffNielsen}. We aim to generalize this method for Anosov surface groups in $SO(n,1)$ in order to define earthquakes.

Guichard-Wienhard \cite{GuichardWienhardAnosov} showed that for $SO(n,1)$ surface groups, being Anosov is equivalent to being convex cocompact (See also the work of Bowditch \cite{BowditchGeomFinite}, Kapovich-Leeb-Porti \cite{KLP2016notes}). We state below a reduced version of their theorem.
\begin{theorem}
\label{thm:GW}
(Guichard-Wienhard) Let $\pi$ be a finitely generated word hyperbolic group and $G$ a real semisimple Lie group of real rank 1. For a representation $\rho:\pi\rightarrow G$, the following are equivalent:
\begin{enumerate}
\item[(i)] $\rho$ is Anosov.
\item[(ii)] There exists a continuous $\rho$-equivariant and injective map $L:\partial_\infty\pi \rightarrow G/P$
\item[(iii)] ker $\rho$ is finite and $\rho(\pi)$ is convex cocompact.
\end{enumerate}
\end{theorem}
Indeed since $S$ is a closed surface, this coincide with the notion of $1$-quasifuchsian group in \cite{KapovichHigherKlein}. When $\pi = \pi_1(S)$ and $G=SO(n,1)$, condition $(ii)$ above implies that there exists a continuous $\rho$-equivariant injective map $L:\partial_\infty(\widetilde{S})\rightarrow \partial_\infty\mathbb{H}^n$ whose image is the limit set $\Lambda$, a quasi-circle in $\partial_\infty\mathbb{H}^n$. Moreover, $\rho(\pi)$ is purely loxodromic, each element of $\rho(\pi)$ has an attracting and a repelling fixed point on $\Lambda$, and the centralizer of this element must contain a 1-parameter group of hyperbolic transformations having the same pair of fixed points. We use these 1-parameter groups to define twist deformations of $\rho$ along weighted simple closed curves, where the weight determines the translation length.

To pass from simple closed curves to measured laminations, we use the geodesic currents view of measured laminations \cite{BonahonCurrents}. Thus a measured lamination is a $\pi_1(S)$-invariant measure on $G(\widetilde{S}) =(\partial_\infty(\widetilde{S})\times\partial_\infty(\widetilde{S}) - \Delta)/\mathbb{Z}_2$ with $0$ self-intersection. Convergence of weighted simple closed curves to a measured lamination becomes convergence of measures on $G(\widetilde{S})$. Together with the homeomorphism $L:\partial_\infty(\widetilde{S})\rightarrow \Lambda$ we can show the following convergence result.
\begin{theorem_} (Corollary \ref{cor:convergence})
For any $\epsilon>0$, $\lambda\in\mathcal{ML}(S)$ an Anosov surface group representation $\rho$ into $SO(n,1)$, there is a neighborhood $U\ni \lambda$ such that for any two weighted simple closed curves $l_1, l_2\in U$, the corresponding representations $\mathcal{E}_{l_1}(\rho)$ and $\mathcal{E}_{l_2}(\rho)$ are $\epsilon$ close.
\end{theorem_}
\noindent This means that we can simply define $\mathcal{E}_\lambda(\rho)$ to be  $\lim_{i\rightarrow\infty}\mathcal{E}_{l_i}(\rho)$ whenever $(l_i)\rightarrow \lambda$, and we have a continuous map from $\mathcal{ML}(S)$ to the space of representations near $\rho$. This is a direct generalization of the classical earthquake on hyperbolic surfaces. But unlike in dimension 2, simple dimension count implies that for $n>2$, these earthquakes cannot take $\rho$ to all nearby points in the moduli space.

\subsection{Implications and further directions}
In \cite{McMullenComplexEq}, McMullen defined a notion of complex earthquake: starting from a fuchsian group $\Gamma \subset PSL(2,\mathbb{R})\subset PSL(2,\mathbb{C})$ it combines classical earthquake and Thurston's bending/grafting deformation and deform the original group to become quasifuchsian. Theorem \ref{thm:commute} and Corollary \ref{cor:convergence} implies that starting from a quasifuchsian $\Gamma$, along any measured lamination, both earthquake and Thurston's bending can be defined and they are commutative. Thus we can combine them into quake-bend deformations with complex parameters like in \cite{BonahonShearBend}. One can then ask whether it is possible to connect 2 quasifuchsian surface groups by 2 quake-bend operations. Thurston showed that we can do that starting from a fuchsian group.
\begin{theorem_} (Thurston, \cite{KTDefo})
The projective grafting map $Gr: \mathcal{ML}(S)\times \mathcal{T}(S)\rightarrow \mathcal{P}(S)$ is a homeomorphism.
\end{theorem_}
\noindent where $\mathcal{T}(S)$ is the Teichmuller space of hyperbolic structures whose holonomy are fuchsian, and  $\mathcal{P}(S)$ is the space of $\mathbb{C}P^1$ structures whose holonomy include all quasifuchsian groups.

We have Anosov surface groups in $G=SO(n,1)$ form an open subset $\mathcal{QF}_G(S)$ of the moduli space $\Hom(\pi_1(S), G)/G$ of representations up to conjugations, but they are in general not a whole component of this space. An interesting question is whether earthquake paths $\mathcal{E}_{t\lambda}(\rho)$ can go outside of the closure of $\mathcal{QF}_G(S)$.

Another interesting direction is the question of Fenchel-Nielsen coordinates for $SO(n,1)$ surface groups, in particular $SO(4,1)$. Tan \cite{TanComplexFN} and Kourouniotis \cite{KComplexLength} constructed complex Fenchel-Nielsen coordinates for the case of $PSL(2,\mathbb{C}) = SO^+(3,1)$. The situation in $SO(4,1)$ is more complicated since in some cases the bending parameter can be any $SO(3)$ rotation, and $SO(3)$ is not abelian. Moreover, by dimension count there may be up to $4$ dimensions of internal parameter for each pair-of-pants in the decomposition, these parameters specify the arrangement of rotation axes. Some work has been done in \cite{Tan4Hexagon} to study pair-of-pants groups in dimension 4, but the whole picture remains mysterious.

\medskip
\noindent{\bf{Acknowledgement.}}
I would like to thank Jean-Marc Schlenker and Virginie Charette for being my mentors during the writing of this article, and I thank my advisor Bill Goldman whose works are such an inspiration.


\section{Deformations along hypersurfaces}

\subsection{The Construction}
\label{sec:TheConstruction}
A manifold $M$ is said to be \emph{aspherical} if its universal cover $\widetilde{M}$ is contractible, or equivalently $\pi_k(M) = 0$ for all $k>1$. For the rest of this section let $M$ be a connected aspherical manifold (possibly with boundary) of dimension at least $2$ with finitely generated fundamental group, and let $\iota: N\hookrightarrow M$ be a connected properly embedded two-sided aspherical hypersurface such that $\iota_*:\pi_1(N)\rightarrow\pi_1(M)$ is injective. 

Let $\rho:\pi_1(M)\rightarrow G$  be a representation into an arbitrary Lie group $G$. We will define deformations of $\rho$ along the hypersurface $N$. An important example is when $M$ is a hyperbolic surface and $N$ is a simple closed geodesic, or more generally when $N$ is a totally geodesic hypersurface in a hyperbolic manifold $M$, in these cases our construction corresponds to Johnson-Millson's bending in \cite{JohnsonMillson}.

We let $\widehat{N}\subset \widetilde{M}$ denote the preimage of $N$ in the universal cover. Then $\widehat{N}$ is a disjoint union of a countable collection $\mathcal{N} =\{N_0, N_1, ...\}$ of connected components, each $N_i$ is a copy of $\widetilde{N}$ which is contractible and of codimension 1 in $\widetilde{M}$. So $\widetilde{M} - N_i$ has two components with a common boundary $N_i$.

\begin{definition}
We define the \emph{dual tree} $T$ to $N\subset M$ to be the tree whose vertices are connected components of $\widetilde{M}-\widehat{N}$ and edges are $N_i$. The adjacency of vertices and edges corresponds to the adjacency of components $\widetilde{M}-\widehat{N}$ with $N_i$'s.
\end{definition}
\noindent Each vertex of the tree $T$ possibly has an infinite number of edges attached. 
Deck transformation  action of $\pi_1(M)$ on $\widetilde{M}$ induce an action on $T$ which is transitive (but not free) on the set of vertices and the set of edges.

Choose a base point  $\tilde{x}_0\in\widetilde{M}$, with a corresponding base point $x_0\in M$ so that $x_0\not\in N$.
For all $A, B \in\pi_1(M, x_0)$, they act on $\widetilde{M}$ as a right action of deck transformations, so in our notation we have $\tilde{x}.(AB) = (\tilde{x}.A).B$ for $\tilde{x}\in\widetilde{M}$.
Now we will define an algebraic twist deformation of the representation $\rho:\pi_1(M, x_0)\rightarrow G$.

Let $y_0\in N$ a base point and choose a path from $x_0$ to $y_0$. By path concatenation this induces an injective homomorphism $\pi_1(N,y_0)\rightarrow \pi_1(M,x_0)$ whose image is a subgroup $T_0\subset \pi_1(M,x_0)$. The path $x_0$ to $y_0$ lifts uniquely to a path $\tilde{x}_0$ to $\tilde{y}_0$ in $\widetilde{M}$. We have $\tilde{y}_0$ is on a component of $\widehat{N}$ and we label this component $N_0$. The action of $T_0$ preserves $N_0$, that is, $N_0.A = N_0$ for all $A\in T_0$. For each component $N_i$ of $\widehat{N}$ we have $N_i = N_0. A_i$, so the subgroup $T_i = A_i^{-1} T_0 A_i$ preserves $N_i$.

The following is the crucial point of our construction. We define a map $\gamma:\mathcal{N}\rightarrow G$ with the following properties:
\begin{itemize}
\item[(1)] $\gamma(N_i)\rho(A) = \rho(A)\gamma(N_i)$ for all $A\in T_i$, that is, $\gamma(N_i)$ is in the centralizer of $\rho(T_i )$ in $G$,
\item[(2)] If $A\in\pi_1(M,x_0)$ such that $N_j = N_i. A$, then $\gamma(N_j) = \rho(A^{-1})\gamma(N_i) \rho(A)$.
\end{itemize}
For the rest of this article, $\gamma$ will be reserved to denote this map. Note that property (2) implies property (1), since $N_i = N_i. T_i$. 
We want to check that if the centralizer of $\rho(T_i)$ is non-trivial then a non-trivial $\gamma$ satisfying (2) exists.

We claim that choosing $\gamma(N_0)\ne 1$ in the centralizer of $\rho( T_0)$ uniquely determines $\gamma(N_i)$ for all $N_i\in\mathcal{N}$. For $N_i\in\mathcal{N}$, suppose that $N_i = N_0.A_i$, then we can use property (2) to define $\gamma(N_i) =  \rho(A_i^{-1})\gamma(N_0) \rho(A_i)$. Note that $A_i$ is not unique for each $N_i$. If we have $N_i = N_0.A_i' = N_0. A_i$ then $N_0.A_i (A_i')^{-1} = N_0$, so $A_i(A'_i)^{-1} \in  T_0 $. Since we chose $\gamma(N_0)$ in the centralizer of $\rho( T_0)$ we have 
$$
\begin{array}{rcl}
\gamma(N_0) & = & \rho(A_i' A_i^{-1})\gamma(N_0)\rho(A_i (A_i')^{-1}) \\
\rho(A_i'^{-1})\gamma(N_0)\rho(A_i') & = & \rho(A_i^{-1})\gamma(N_0)\rho(A_i)
\end{array}
$$
So $\gamma(N_i)$ is well-defined, not depending on different choices of $A_i$.

Now suppose that $N_j = N_i. B$, and we have $N_i = N_0. A_i$. So we can let $A_j =A_iB$, thus $B=A_i^{-1} A_j$.
We have 
$$\begin{array}{rcl}
\gamma(N_j) & = & \rho(A_j^{-1})\gamma(N_0)\rho(A_j) \\
 & = &  \rho(A_j^{-1} A_i)\rho(A_i^{-1}) \gamma(N_0)\rho(A_i)\rho(A_i^{-1} A_j) \\
 & = & \rho(B^{-1})\gamma(N_i) \rho(B).
\end{array}
$$
This is property (2). Therefore $\gamma$ as above is well-defined and satisfies (1) and (2).

We are now ready to define the twist deformation of representation $\rho$ with respect to $\gamma$. First we choose an orientation of $N\subset M$, which induces orientation of each $N_i\subset\widetilde{M}$.
\begin{definition}
\label{def:Earthquake}
Let $\gamma:\mathcal{N}\rightarrow G$ be defined as above. Given $A\in\pi_1(M,x_0)$, a (minimal) path in $\widetilde{M}$ with direction from $\tilde{x}_0$ to $\tilde{x}_0 . A$ will cross a sequence of hypersurfaces in $\mathcal{N}$ which we name $N_{a_1},...N_{a_k}$ in order. We define 
$$\mathcal{E}_{N,\gamma}(\rho)(A) = \mathcal{E}(\rho)(A) := \rho(A) \left[\gamma(N_{a_k})^{s_k}...\gamma(N_{a_1})^{s_1}\right]$$
where each $s_i = \pm 1$ and equals the intersection sign between $N_i$ and the directed path  $\tilde{x}_0$ to $\tilde{x}_0 . A$.
\end{definition}
\noindent We say that the deformation is non-trivial if $\gamma(N_i)\ne 1$.
\begin{remark}
\label{rem:non-minimal}
Note that it is not necessary for the path from $\tilde{x}_0$ to $\tilde{x}_0 A$ to be minimal since any overlapping would be cancelled out in the expression for $\mathcal{E}(\rho)(A)$, this is because of the tree structure of $\widetilde{M}-\widehat{N}$. The intersection sign between $N_i$ and the path depends on the chosen orientation of $M$ and $N$. There is a dual definition of $\mathcal{E}(\rho)$ obtained by changing the orientation of $M$, or changing the orientation of $N\subset M$. Both are equivalent to switching between a ``left'' and a ``right'' Fenchel-Nielsen twist for the case of surface groups in $G=\Isom^+(\mathbb{H}^2)$.
\end{remark}

From the above definition we have a well-defined map $\mathcal{E}(\rho):\pi_1(M,x_0)\rightarrow G$, because a minimal path  from $\tilde{x}_0$ to $\tilde{x}_0 . A$ induces a minimal path in the dual tree $T$ and which is unique, so the sequence of lines $N_{a_1},...,N_{a_k}$ is uniquely determined. The following proposition shows that $\mathcal{E}(\rho):\pi_1(M,x_0)\rightarrow G$ is a homomorphism.
\begin{proposition}
Let $A,B\in\pi_1(M,x_0)$ and $\mathcal{E}(\rho)$ as defined above. Then 
\begin{enumerate}[(i)]
\item $\mathcal{E}(\rho)(AB) = \mathcal{E}(\rho)(A)\mathcal{E}(\rho)(B),$
\item $\mathcal{E}(\rho)(A^{-1}) = \mathcal{E}(\rho)(A)^{-1}.$
\end{enumerate}
Therefore $\mathcal{E}(\rho)$ is a homomorphism.
\end{proposition}
\begin{proof}\hfill

\noindent $(i)$  We use notations as before in this section. Let $N_{a_1},...,N_{a_p}$ be the sequence of hypersurfaces between $\tilde{x}_0$ and $\tilde{x}_0 A$, and $N_{b_1},...,N_{b_q}$ be the sequence of hypersurfaces between $\tilde{x}_0$ and $\tilde{x}_0 B$, and let $N_{c_1},...,N_{c_r}$ be the sequence of hypersurfaces between $\tilde{x}_0$ and $\tilde{x}_0 AB$. Suppose that the definition of $\mathcal{E}(\rho)$ gives us

$\mathcal{E}(\rho)(A) = \rho(A)\alpha_p...\alpha_1$

$\mathcal{E}(\rho)(B) = \rho(B)\beta_q...\beta_1$

$\mathcal{E}(\rho)(AB) = \rho(A)\rho(B)\delta_r ...\delta_1$

\noindent Where $\alpha_i = \gamma(N_{a_i})^{\pm 1}$, $\beta_i = \gamma(N_{b_i})^{\pm 1}$, and $\delta_i = \gamma(N_{c_i})^{\pm 1}$, with the signs determined by the corresponding intersection signs.

\begin{figure} [h] 
 \begin{center}
    \includegraphics[width=10cm]{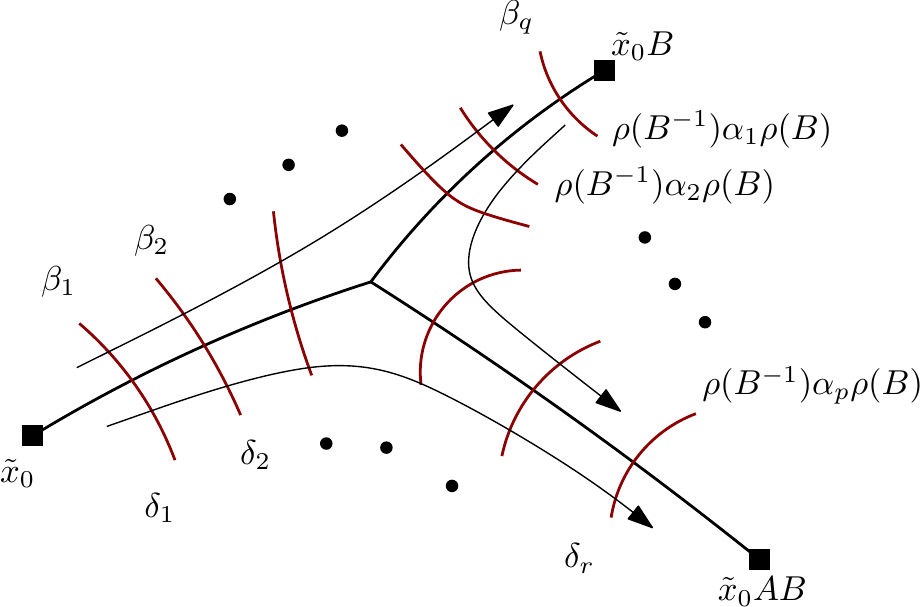}
    \caption{}
    \label{fig:HomomorphismLemma}
  \end{center}
\end{figure}

We have 

$\mathcal{E}(\rho)(A)\mathcal{E}(\rho)(B)$

$=\rho(A)\alpha_p...\alpha_1\rho(B)\beta_q...\beta_1
$

$ = \rho(A)\rho(B) [\underbrace{(\rho(B^{-1})\alpha_p\rho(B))}_{}... \underbrace{(\rho(B^{-1})\alpha_1\rho(B))}_{} \beta_q ... \beta_1 ]$

Indeed $\beta_1,...,\beta_q$ are elements in $G$ associated with hypersurfaces $N_{b_1},...,N_{b_q}$ which are between $\tilde{x}_0$ and $\tilde{x}_0. B$, and $\rho(B^{-1})\alpha_1\rho(B),...,\rho(B^{-1})\alpha_p\rho(B)$ are elements associated with hypersurfaces $N_{a_1}.B, ..., N_{a_p}.B$ which are between $\tilde{x}_0 .B$ and $\tilde{x}_0 .AB$. On the other hand, $\delta_1,...,\delta_r$ are elements associated with the sequence of hypersurfaces between $\tilde{x}_0$ and $\tilde{x}_0 .AB$. 
Consider the tree structure of the dual tree $T$ (see Figure \ref{fig:HomomorphismLemma}). Indeed the difference between a minimal path from $\tilde{x}_0$ to $\tilde{x}_0 AB$ and a concatenation of 2 minimal paths $\tilde{x}_0$ to $\tilde{x}_0 B$ to $\tilde{x}_0 AB$ is a segment of back tracking, and along that segment which crosses let's say $m$ hypersurfaces, we must have cancellation between $\rho(B^{-1})\alpha_m\rho(B)...\rho(B^{-1})\alpha_1\rho(B)$ and $\beta_q ... \beta_{q-m+1}$.

So
$$ [{\rho(B^{-1})\alpha_p\rho(B)}... {\rho(B^{-1})\alpha_1\rho(B)}]\beta_q ... \beta_1 
= \delta_r...\delta_1.
$$
Therefore $\mathcal{E}(\rho)(A)\mathcal{E}(\rho)(B) = \mathcal{E}(\rho)(AB)$.

\medskip

\noindent$(ii)$ A similar argument as above shows that

$\mathcal{E}(\rho)(A^{-1})$

 $ = \rho(A^{-1} )\underbrace{\rho(A)\alpha_1^{-1}\rho(A)^{-1}}_{}...\underbrace{\rho(A)\alpha_k^{-1}\rho(A)^{-1}}_{}$

$= \alpha_1^{-1}...\alpha_k^{-1}\rho(A^{-1})$

$= \mathcal{E}(\rho)(A)^{-1} .$

\end{proof}
For the case of simple closed curves in surfaces we have the following.
\begin{proposition}
Let $M$ be a closed surface of genus $g>1$ and let $N, L$ be disjoint simple closed curves in distinct homotopy classes, and let $\rho:\pi_1(M)\rightarrow G$ be a representation such that $\rho(\pi_1(N))$ and $\rho(\pi_1(L))$ have non-trivial centralizers. Then non-trivial deformations (as constructed in definition \ref{def:Earthquake}) of $\rho$ along $N$ and $L$ are distinct.
\end{proposition}
\begin{proof}
We cut the surface $M$ along $N$ and obtain a surface $M'$ with boundary which still contain $L$, so we can choose a base point $x_0\in M$ in the same component containing $L$ (if $N$ is a separating curve). Note that the representation $\mathcal{E}_\rho$ only changes by a conjugation as the base point changes. Now we can choose $A\in\pi_1(M,x_0)$ represented by a curve intersecting $L$ and not touching the boundary. Thus $\mathcal{E}_N(\rho)(A) = \rho(A)$ while $\mathcal{E}_L(\rho)(A) \ne \rho(A)$. Thus $\mathcal{E}_N(\rho)\ne\mathcal{E}_L(\rho)$.
\end{proof}

\subsection{Deformation along multi-hypersurfaces}
\label{sec:multi-hyp}

In this section we will define deformations along a set of disjoint hypersurfaces - the generalization of deforming a surface group along multicurves. A theorem on commutativity of deformations will be shown.

Let $L = N_{(1)}\cup ...\cup N_{(n)} $ be a collection of mutually disjoint connected hypersurfaces of $M$, each satisfies properties of $N$ in previous section. Moreover we require that no pair $N_{(i)}, N_{(j)}$ are homotopic.  Let
 $$\mathcal{L} = \{ L_1, L_2, ...\} = \bigcup_{i=1}^n \{N_{(i), 1}, N_{(i), 2}, ... \}$$ be the collection of components of pre-images of $N_{(1)},..., N_{(n)}$ in $\widetilde{M}$. As before, $T_i\subset \pi_1(M,x_0)$ is the subgroup preserving $L_i$.

 Let $\gamma:\mathcal{L}\rightarrow G$ be such that
 \begin{itemize}
\item[(1)] $\gamma(L_i)\rho(A) = \rho(A)\gamma(L_i)$ for all $A\in T_i$
\item[(2)] If $A\in\pi_1(M,x_0)$ such that $L_j = L_i. A$, then $\gamma(L_j) = \rho(A^{-1})\gamma(L_i) \rho(A)$.
\end{itemize}
Suppose $L_1,...,L_n$ are $N_{(1),0}, ..., N_{(n),0}$ in order. Then indeed, $\gamma$ is determined by $\gamma(L_1),..., \gamma(L_n)$. 
 
 The following definition and proposition are straight forward generalization of previous section.
\begin{definition}
\label{def:EarthquakeMulti}
Let  $L = N_{(1)}\cup ...\cup N_{(n)} $ and  $\gamma:\mathcal{L}\rightarrow G$ be defined as above. Given $A\in\pi_1(M,x_0)$, a (minimal) path in $\widetilde{M}$ with direction from $\tilde{x}_0$ to $\tilde{x}_0 . A$ will cross a sequence of hypersurfaces in $\mathcal{L}$ which we name $L_{a_1},... , L_{a_k}$ in order. We define 
$$\mathcal{E}_{L,\gamma}(\rho)(A) = \mathcal{E}(\rho)(A) := \rho(A) \left[\gamma(L_{a_k})^{s_k}...\gamma(L_{a_1})^{s_1}\right]$$
where each $s_i = \pm 1$ and equals the intersection sign between $L_i$ and the directed path  $\tilde{x}_0$ to $\tilde{x}_0 . A$.
\end{definition}
\begin{proposition}
$\mathcal{E}_{L,\gamma}(\rho)(AB) = \mathcal{E}_{L,\gamma}(\rho)(A) \mathcal{E}_{L,\gamma}(\rho)(B)$. 
Deformation along multiple disjoint hypersurfaces gives $\mathcal{E}_{L,\gamma}(\rho)$ a homomorphism.
\end{proposition}

We will need the following discussion to prove commutativity of deformations. Let $x_0, x'_0$ be different base points on $M$, $x_0, x'_0\not\in L$. Choose a path from $x_0$ to $x'_0$ which determine 
a lift $\tilde{x}_0$ to $\tilde{x}'_0$ on $\widetilde{M}$. We will deform $\rho:\pi_1(M, x_0)\rightarrow G$ in two different ways: the first using definition \ref{def:EarthquakeMulti} and base point $x_0$ resulting in $\mathcal{E}_{L,\gamma}(\rho)$; the second way of deforming is by using $\tilde{x}'_0$ as base point. Let $\mathcal{E}'_{L,\gamma}(\rho)$ be the resulting representation obtained by using $\tilde{x}'_0$ as the base point as follows: for $A\in\pi_1(M, x_0)$, we use the (possibly non-minimal) path from $\tilde{x}'_0$ to $\tilde{x}'_0 A$ obtained by going from $\tilde{x}'_0$ to $\tilde{x}_0$, then to $\tilde{x}_0 A$, then to $\tilde{x}'_0 A$. By remark \ref{rem:non-minimal} we see that taking a non-minimal path does not change the deformation result. We will now show that the difference between $\mathcal{E}'_{L,\gamma}(\rho)$ and $\mathcal{E}_{L,\gamma}(\rho)$ is conjugating by some transformation in $G$.
\begin{lemma}
\label{lem:x0prime}
Let the minimal path from $\tilde{x}_0$ to $\tilde{x}'_0$ cross $L_1,...,L_k$ with intersection sign $s_1,..., s_k$. Then
$$ \mathcal{E'}_{L,\gamma}(\rho)(A) =  \left(\prod_{i=k}^1 \gamma(L_i)^{s_i}\right) \mathcal{E}_{L,\gamma}(\rho)(A) \left(\prod_{i=k}^1 \gamma(L_i)^{s_i}\right)^{-1}$$
\end{lemma}
\begin{proof}
 $L_1,...,L_k$ are the hypersurfaces between $\tilde{x}_0$ and $\tilde{x}'_0$ in order, so $L_1 A,... , L_k A$ are between $\tilde{x}_0 A$ and $\tilde{x}'_0 A$. Let $L_{a_1},... , L_{a_r}$ be the hypersurfaces between $\tilde{x}_0$ and $\tilde{x}_0 A$. We have the path from $\tilde{x}'_0$ to $\tilde{x}_0$ to $\tilde{x}_0 A$ to $\tilde{x}'_0 A$ crosses the following hypersurfaces 
$$L_k, ... , L_1, L_{a_1}, ... , L_{a_r}, L_1 A, ... , L_k A.$$
So using property (2) of $\gamma$ we get
 $$\mathcal{E}'_{L,\gamma}(\rho)(A) = \rho(A)\left( \prod_{i=k}^1 \rho(A^{-1}) \gamma(L_i)^{s_i}\rho(A)\right) \left(\prod_{i=r}^1\gamma(L_{a_i})^{s_{a_i}}\right)\left(\prod_{i=1}^k \gamma(L_i)^{-s_i}\right)
 $$
 $$ =  \left(\prod_{i=k}^1 \gamma(L_i)^{s_i}\right)\rho(A)\left(\prod_{i=r}^1\gamma(L_{a_i})^{s_{a_i}}\right) \left(\prod_{i=k}^1 \gamma(L_i)^{s_i}\right)^{-1}
 $$
 and we get the result.
\end{proof}
\begin{lemma}
\label{lem:conjugateTi} Let $L_a\in\mathcal{L}$ be hypersurface and $T_a\subset \pi_1(M,x_0)$ be the subgroup preserving it. Suppose the minimal path from $\tilde{x}_0$ to $L_a$ crosses $L_1,..., L_k$ in order (not including $L_a$), then
$$ \mathcal{E}_{L,\gamma}(\rho)(T_a) = \left(\prod_{i=k}^1 \gamma(L_i)^{s_i}\right)^{-1} \rho(T_a) \left(\prod_{i=k}^1 \gamma(L_i)^{s_i}\right)$$
\end{lemma}
\begin{proof}
We can choose $\tilde{x}'_0$ near $L_a$ so that the minimal path from $\tilde{x}_0$ to $\tilde{x}'_0$ cross $L_1,...,L_k$. That is, $\tilde{x}'_0$ is in the component $C$ of $\widetilde{M} - \bigcup_{i=1}^\infty L_i$ adjacent to $L_a$ and closest to $\tilde{x}_0$. So $\tilde{x}'_0 A$ is still in $C$ for all $A\in T_a$, which means $\mathcal{E}'_{L,\gamma}(\rho)(A) = \rho(A)$ for all $A\in T_a$ (because of path independence for definition \ref{def:EarthquakeMulti}). Applying lemma \ref{lem:x0prime} we get $ \mathcal{E}_{L,\gamma}(\rho)(A) =  \left(\prod_{i=k}^1 \gamma(L_i)^{s_i}\right)^{-1} \rho(A) \left(\prod_{i=k}^1 \gamma(L_i)^{s_i}\right)$
\end{proof}
The following theorem implies commutativity of deforming along disjoint (multi-) hypersurfaces, and it also shows that choosing $\gamma$ values in a 1-parameter subgroup of the centralizer induces a flow path of representations.
\begin{theorem}
\label{thm:commute}
Let $\gamma$ and $\alpha$ be maps $\mathcal{L}\rightarrow G$ satisfying (1) and (2), and so that $\gamma(L_i)\alpha(L_i) = \alpha(L_i)\gamma(L_i)$. Then
$$\mathcal{E}_{L,\gamma}(\mathcal{E}_{L,\alpha}(\rho)) = \mathcal{E}_{L,\alpha}(\mathcal{E}_{L,\gamma}(\rho)).$$
\end{theorem}
\begin{proof}
Note that $\gamma$ on the left hand side is different from the one on the right because the starting representations are different. When distinction is required we write $\gamma_\rho$ for the $\gamma$ on the right hand side and
$\gamma_{\mathcal{E}}$ for the left hand side, similarly for $\alpha$. The equation can be written more unambiguously as
$$
\mathcal{E}_{L,\gamma_\mathcal{E}}(\mathcal{E}_{L,\alpha_\rho}(\rho)) = \mathcal{E}_{L,\alpha_\mathcal{E}}(\mathcal{E}_{L,\gamma_\rho}(\rho)).
$$
Note that we have yet to define $\gamma_\mathcal{E}$ and $\alpha_\mathcal{E}$. Since, by lemma \ref{lem:conjugateTi}, $\mathcal{E}_{L,\alpha}(\rho)(T_i) = \xi_i^{-1} \rho(T_i) \xi_i$ for some transformation $\xi_i$, we can naturally define $\gamma_\mathcal{E}$ so that $\gamma_\mathcal{E}(L_i) = \xi_i^{-1}\gamma_\rho(L_i)\xi_i$. Similarly can relate $\alpha_\rho$ and $\alpha_\mathcal{E}$ this way.

Let $A\in\pi_1(M, x_0)$, let $L_1,..., L_k$ be the hypersurfaces between $\tilde{x}_0$ and $\tilde{x}_0 A$ and $T_i$ be the subgroup preserving $L_i$, for $i=1,..,k$. Since $L_1,...,L_{i-1}$ is between $\tilde{x}_0$ and $L_i$, by lemma \ref{lem:conjugateTi} we have
$$\mathcal{E}_{L,\alpha}(\rho)(T_i) = \alpha_\rho(L_1)^{-s_1}...\alpha_\rho(L_{i-1})^{-s_{i-1}}  \rho(T_i) \alpha_\rho(L_{i-1})^{s_{i-1}}  ... \alpha_\rho(L_1)^{s_1}
$$
So for $i=1,...k$
\begin{equation}
\label{eq:gammaConj}
\gamma_\mathcal{E}(L_i) =  \alpha_\rho(L_1)^{-s_1}...\alpha_\rho(L_{i-1})^{-s_{i-1}}  \gamma_\rho(L_i)  \alpha_\rho(L_{i-1})^{s_{i-1}} ... \alpha_\rho(L_1)^{s_1}.
\end{equation}
Moreoever, by definition
\begin{equation}
\label{eq:EcomposeE}
\mathcal{E}_{L,\gamma}(\mathcal{E}_{L,\alpha}(\rho))(A) = \rho(A) \left[ \alpha_\rho(L_k)^{s_k} ... \alpha_\rho(L_1)^{s_1} \gamma_\mathcal{E}(L_k)^{s_k} ...\gamma_\mathcal{E}(L_1)^{s_1} \right].
\end{equation}
From (\ref{eq:gammaConj}) and (\ref{eq:EcomposeE}) we get 
$$
\mathcal{E}_{L,\gamma}(\mathcal{E}_{L,\alpha}(\rho))(A) = \rho(A) \left[ \alpha_\rho(L_k)^{s_k}\gamma_\rho(L_k)^{s_k}...\alpha_\rho(L_1)^{s_1}\gamma_\rho(L_1)^{s_1}\right].
$$ 
An analogous formula can be shown for
$\mathcal{E}_{L,\alpha}(\mathcal{E}_{L,\gamma}(\rho))(A) $, and together with the commutativity assumption $\alpha_\rho(L_i)\gamma_\rho(L_i) = \gamma_\rho(L_i)\alpha_\rho(L_i)$, the theorem follows.
\end{proof}

\begin{corollary}
Deformations along disjoint hypersurfaces commute.
\end{corollary}

\subsection{Infinitesimal deformation}

Let $\mathfrak{g}$ be the lie algebra of $G$, let $\pi = \pi_1(M)$.
\begin{definition}
Suppose $s:(-\epsilon,\epsilon)\rightarrow G$ is an analytic path such that $s(0)=1$, that is, $s(t) = \exp(a_1 t + a_2 t^2 +...)$ for $a_i\in\mathfrak{g}$; and suppose that $a_1\ne 0$. Then we let $\frac{d}{dt}|_{t=0} s(t) = a_1$ and it will be called \emph{the derivative of $s(t)$}. We may also write $\dot{s}(0)=a_1$.
\end{definition}
With the notation as in previous section, suppose we have a 1-parameter subgroup $ e^{t X}$ (where $X\in\mathfrak{g}$) contained in the centralizer of $\rho(N_0)$. Thus we have a family of choices for $\gamma(N_0)$, we write $\gamma^t(N_0) = e^{tX}$. Applying the above construction we get a differentiable 1-parameter family of deformations of the representation $\rho\in\Hom(\pi,G)$. 

Composition with the adjoint representation gives us $\pi\stackrel{\rho}{\rightarrow} G\stackrel{Ad}{\rightarrow}\Aut(\mathfrak{g})$. An infinitesimal deformation of $\Ad\rho$ is given by a 1-cocycle $u:\pi\rightarrow\mathfrak{g}$. See \cite{GoldmanSymplecticNature}.

Let $A\in\pi$. For simplicity of notation suppose we have a family of deformations as in previous section given by $\mathcal{E}^t(\rho)(A) =\rho(A)\gamma^t_k...\gamma^t_1$ where each $\gamma^t_i = e^{tX_i}$ is a 1-parameter group. Of course the collection $X_1,..., X_k\in\mathfrak{g}$ depends on $A$. Let $\mathcal{E}^t(\rho)(A) = \rho(A)\exp( tu(A) + O(t^2))$ for $u(A)\in\mathfrak{g}$, in other words we have 
\begin{definition} For a smooth deformation $\mathcal{E}^t(\rho)$ of $\rho$ such that $\mathcal{E}^0(\rho) = \rho$, the map $u:\pi_1(M)\rightarrow \mathfrak{g}$ defined by
  $$u(A):=\frac{d}{dt}|_{t=0}\left[\rho(A)^{-1}\mathcal{E}^t(\rho)(A) \right]$$
is called the \emph{infinitesimal deformation corresponding to $\mathcal{E}^t(\rho)$}.
\end{definition}
Indeed $u$ is a 1-cocycle corresponding to the 1-parameter family of representations $\mathcal{E}^t(\rho)$. We have 
$$
\begin{array}{rcl}
\exp(tu(A) + O(t^2))&  = &  \gamma^t_k...\gamma^t_1 \\
&  = & e^{tX_k}...e^{tX_1} \\
 & = & e^{t(X_k +... + X_1) + O(t^2)}
\end{array}
$$
where by the Baker--Campbell--Hausdorff formula, the coefficients for higher order $t$-terms on the right hand side include various combinations nested of lie bracket $[,]$ between $X_1,...,X_k$. Therefore considering the linear coefficients of $t$ we get
\begin{remark}
\label{rem:u(A)}
$u(A) = X_1 +  ... + X_k$.
\end{remark}

\section{Algebraic earthquake along measured laminations}
We now switch our attention to the case where $S$ is a closed hyperbolic surface, and Lie group $G=SO(n,1) =\Isom(\mathbb{H}^n)$, with Lie algebra $\mathfrak{g} = \mathfrak{so}(n,1)$. For the rest of this section, let $\rho:\pi_1(S)\rightarrow SO(n,1)$ be an Anosov representation. In this section we will define twist deformations $\rho$ along weighted simple closed curves, and then earthquakes along measured laminations on $S$. We have
\begin{theorem}
\label{thm:GW}
(Guichard-Wienhard \cite{GuichardWienhardAnosov}) Let $\pi$ be a finitely generated word hyperbolic group and $G$ a real semisimple Lie group of real rank 1. For a representation $\rho:\pi\rightarrow G$, the following are equivalent:
\begin{enumerate}
\item[(i)] $\rho$ is Anosov.
\item[(ii)] There exists a continuous $\rho$-equivariant and injective map $L:\partial_\infty\pi \rightarrow G/P$
\item[(iii)] ker $\rho$ is finite and $\rho(\pi)$ is convex cocompact.
\end{enumerate}
\end{theorem}
Considering that $G=SO(n,1)$ is of rank 1, and $\pi_1(S)$ is word hyperbolic, we will use condition $(ii)$ above as the definition of Anosov representations in our case. Thus there is a $\rho$-equivariant homeomorphism $L: \mathbb{S}^1_\infty\stackrel{}{\rightarrow} \Lambda$ where $\mathbb{S}^1_\infty$ is the boundary at infinity of $\widetilde{S}$ and $\Lambda$ is the limit set of $\rho(\pi_1(S))$. For the rest of this article $L$ is reserved to denote this limit set map.

\subsection{Measured laminations}
\label{subsec:MLS}

Let $A\in\pi_1(S)$, $\lambda\in\mathcal{ML}(S)$, $\tilde{x_0}\in\widetilde{S}$. Following \cite{AramayonaLeininger} and \cite{BonahonCurrents}, we can view $\lambda$ as an $\pi_1(S)$-invariant measure on $G(\tilde{S}) =( {\mathbb{S}^1_\infty}\times{\mathbb{S}^1_\infty} - \Delta) / \mathbb{Z}_2$ where $\mathbb{Z}_2$ acts by swapping the two coordinates. Indeed $G(\tilde{S})$ is the space of unoriented geodesic on $\tilde{S}$. 

Also considering the geodesic arc $\mathcal{A}$ from $\tilde{x_0}$ to $\tilde{x_0}A$, and we can define $\lambda|_A$ to be the measure on the space of geodesics intersecting $\mathcal{A}$ transversely. We have the support of $\lambda|_A$ is contained inside the interior of a compact set $I_1 \times I_2$ where $I_i$'s are disjoint closed intervals of $\mathbb{S}^1_\infty$, indeed $I_1, I_2$ are separated by (neighborhoods of) the end points of the extended geodesic containing $\mathcal{A}$. If a sequence $(\lambda_i)$ converges to $\lambda$ in $\mathcal{ML}(S)$ then $\lambda_i|_A$ converges to $\lambda|_A$ in the weak topology on the space of measures on $I_1\times I_2$.

Given $\rho:\pi_1(S)\rightarrow SO(n,1)$ an Anosov representation, recall that the homeomorphism $L: \mathbb{S}^1_{\infty} \rightarrow \Lambda_\rho$ gives us $G(\tilde{S}) \stackrel{\sim}{ \rightarrow} ({\Lambda_\rho}\times {\Lambda_\rho} -\Delta)/\mathbb{Z}_2$ and in particular $I_1\times I_2 \stackrel{\sim}{\rightarrow}L(I_1)\times L(I_2)$. Thus any measure on $I_1\times I_2$ pushes forward to a measure on $L(I_1)\times L(I_2)$. In particular $\lambda|_A$ on $I_1\times I_2$ induces $L_*\lambda|_A$ a measure on $L(I_1)\times L(I_2)$. 

\subsection{Deforming Anosov representations along weighted simple closed curves}
\label{subsec:twist}
Suppose we have a weighted simple closed geodesic $l\subset S$, and $\rho:\pi_1(S)\rightarrow SO(n,1)$ an Anosov representation. We will now define a (right) twist of $\rho$ along $l$. Following the construction in section \ref{sec:TheConstruction}, let $\tilde{x}_0$ be a base point in $\widetilde{S}$ not lying on a lift of $l$. We choose an orientation for $S$ and an orientation for $l$. The way we construct $\gamma$, it will turn out that the orientation of $l$ does not matter.

\begin{notation}
For $x,y\in\partial_\infty\mathbb{H}^n$, $x\ne y$, $t\in\mathbb{R}$, let $H(x,y,t)$ be the hyperbolic transformation (loxodromic without rotation) fixing $x,y$ with translation length $t$ in the direction from $x$ to $y$. So $x$ is the repelling and $y$ is the attracting fixed point.
\end{notation}

We have $H(x,y,t)$ preserve the geodesic in $\mathbb{H}^n$ whose end-points are $x,y$, and it translate along this geodesic a distance $t$ from $x$ to $y$. 

 Let $l_0$ be a lift of $l$ in the universal cover $\widetilde{S}$, and suppose $T_0\subset \pi_1(S)$ is the subgroup preserving $l_0$ under deck transform action. Then $\rho(T_0)$ is an infinite cyclic group, $\rho(T_0) = \langle \tau_0 \rangle$. Moreover by properties of Anosov representations in $SO(n,1)$, we have $\tau_0$ is loxodromic, thus $\tau_0$ can be written uniquely as a composition $\tau_0 = \sigma_0\theta_0$ such that $\sigma_0 = H(p_0,q_0,t_0)$ is a  hyperbolic transformation in the direction of $l_0$, and $\theta_0$ is an elliptic transformation with $p_0, q_0$ among its fixed points. It's important to note that we choose the generator $\tau_0$ such that $L^{-1}(p_0), L^{-1}(q_0)\in\mathbb{S}^1_\infty$  is respectively the starting and ending point of the directed infinite geodesic $l_0$.
 
 We have $\sigma_0$ and $\theta_0$ commutes with $H(p_0,q_0,t)$ for any $t\in\mathbb{R}$. (This is because up to conjugation $p_0 = 0$ and $q_0 = \infty$ in $\mathbb{R}^{n-1}\cup\{\infty\} = \partial_\infty\mathbb{H}^n$, so $H(p_0,q_0,t)$ acts as scalar multiplication, $\theta_0$ acts as an $SO(n-1)$ rotation on $\mathbb{R}^{n-1}$.) So we can choose 
 $$\gamma^t(l_0) = H(p_0,q_0, tw)$$ where $w$ is the weight of $l$ and $t\in\mathbb{R}^+$, and indeed $\gamma^t (l_0)$ commutes with $\tau_0$ as required by condition (1) in section \ref{sec:TheConstruction}, and thus this choice of $\gamma^t(l_0)$ gives us $\gamma^t(l_i) = H(p_i,q_i,tw)$ for any lift $l_i$. Note that we also have $L^{-1}(p_i), L^{-1}(q_i)$ are respectively the starting and ending points of $l_i$. Following Definition \ref{def:Earthquake} we have constructed a 1-parameter family of algebraic twist deformations of $\rho$.
\begin{definition}
\label{notation:defo}
Let $\mathcal{E}^t_l(\rho)$ be the 1-parameter family of representations obtained by deforming of $\rho$ along the weighted simple closed curve $l$ by choosing $\gamma^t(l_0) = H(p_0,q_0, tw)$ as above. We simply write $\mathcal{E}_l(\rho)$ for when $t=1$.
\end{definition}

\begin{figure} [h] 
 \begin{center}
    \includegraphics[width=5.5cm]{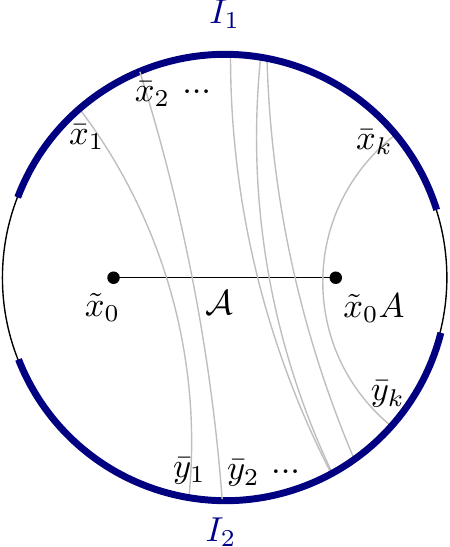}
    \label{fig:I1I2}
    \caption{$I_1\times I_2\subset {\mathbb{S}^1_\infty}^{(2)}$. Here $\bar{x}_i, \bar{y}_i$ are $L^{-1}(x_i), L^{-1}(y_i)$ respectively.}
  \end{center}
\end{figure}

Considering the boundary at infinity $\mathbb{S}^1_\infty = \partial_\infty\widetilde{S}$ which is mapped homeomorphically to the limit set $\Lambda_\rho$, the picture of these deformations is very similar to the classical Fenchel-Nielsen twist deformation of a hyperbolic surface. Let $A\in\pi_1(S)$ and let $\mathcal{A}$ be the geodesic arc from $\tilde{x}_0$ to $\tilde{x}_0 A$. By definition we have 
$$\rho(A)^{-1}\mathcal{E}^t_l(\rho)(A) = \gamma^t(l_k)^{s_k} ...\gamma^t(l_1)^{s_1} = \prod_{i=k}^1 H(x_i,y_i,tw)$$
if we let $\gamma^t(l_i)^{s_i} = H(x_i,y_i,tw)$. Then indeed $L^{-1}(x_i), L^{-1}(y_i)$ are the end points of the lift $l_i$. If we changed the orientation of $l$, the signs $s_1,...,s_k$ would be flipped, but we would also invert the choice of $\gamma^t(l_0)$ and hence of each $\gamma^t(l_i)$, because $H(y_0,x_0,t) = H(x_0,y_0,t)^{-1}$, thus we would keep the deformation the same.

We will now show that all the repelling fixed points $x_i$'s are on one side of $\mathcal{A}$ and the attracting fixed points $y_i$'s are on the otherside. Let $x_*, y_*\in\mathbb{S}^1_\infty$ be the endpoints of the bi-infinite geodesic that is $\mathcal{A}$ extended. Then $L^{-1}(x_i), L^{-1}(y_i)$ are in $\mathbb{S}^1_\infty - \{L^{-1}(x_*), L^{-1}(y_*) \}$ which is the union of two disjoint open intervals $I^o_1, I^o_2$. We name these intervals such that a geodesic from $I^o_1$ to $I^o_2$ intersects $\mathcal{A}$ at a positive intersection point. (Again refer to definition \ref{def:Earthquake}.) This is because if $l_i$ is from $I^o_1$ to $I^o_2$ then $s_i = 1$, so $\gamma^t(l_i) = H(x_i,y_i,tw)$ and so $L^{-1}(x_i)$ is the starting point of $l_i$ which is in $I^o_1$. If $l_j$ is from $I^o_2$ to $I^o_1$ then $s_j = -1$, so $\gamma^t(l_j)^{-1} = H(x_j,y_j,tw)$  so $\gamma^t(l_j) = H(y_j,x_j,tw)$ which makes $L^{-1}(y_j)$ the starting point of $l_j$ which is in $I^o_2$. 
\begin{remark} $x_i\in L(I^o_1)$ and $y_i\in L(I^o_2)$ for $i=1, ... , k$.
\end{remark}
\subsection{Proof of convergence}

For the rest of this section let $\lambda\in\mathcal{ML}(S)$ and $A\in\pi_1(S)$. Our goal is to prove that if a sequence of weighted simple closed curves $(l_i)$ converges to $\lambda$, then $(\mathcal{E}_{l_i}(\rho)(A))$ converges. The limit can then be defined to be $\mathcal{E}_{\lambda}(\rho)(A)$. Our approach for the proof of convergence loosely follow Kerckhoff's \cite{KerckhoffNielsen}.

\begin{notation}
For $x$ a point in a metric space $X$, we denote $B_r(x)$ the open ball of radius $r$ centered at $x$.
\end{notation}
\begin{notation}For a subset $X$ of a topological space, we denote $cl(X)$ the closure of $X$.
\end{notation}
\begin{notation}
Let $X$ be any set, we denote $$X^{(2)}:= ((X\times X) - \{(x,x)| x\in X\})/\mathbb{Z}_2$$
the set of unordered distinct pairs of points in $X$.
\end{notation}
\begin{notation}
Let $v(x,y) \in\mathfrak{g}$ be such that $e^{t.v(x,y)} = H(x,y,t)$ for all $t\in\mathbb{R}$.
\end{notation}
 Note that for $w\in\mathbb{R}^+$ we have $$\frac{d}{dt}|_{t=0}H(x,y,tw) = w.v(x,y).$$

\begin{remark}
Let $\mathbb{R}^{n-1}$ be the stereographic projection coordinate of $\partial_\infty\mathbb{H}^n - \{\infty \}$, and let $x,y\in\mathbb{R}^{n-1}$ . Then $H(x,y,t)$ varies analytically in $SO(n,1)$ as $x,y,t$ vary. That is, the map $H:(\partial_\infty\mathbb{H}^n)^{(2)}\times\mathbb{R} \rightarrow G$ is analytic, where $\partial_\infty\mathbb{H}^n$ is represented by either $\mathbb{R}^{n-1}\cup\{\infty\} $ or $S^{n-1}$. We also have $v:(\partial_\infty\mathbb{H}^n)^{(2)}\rightarrow\mathfrak{g}$ is analytic.
\end{remark}
From this point we work with a chosen positive definite inner product on $\mathfrak{g}$, the corresponding left invariant metric $d_G$ on $G$, and the standard metric $d_S$ on  $S^{n-1} = \partial_\infty\mathbb{H}^n$ which induces the product metric on ${S^{n-1}}^{(2)}$. 
\begin{lemma}
\label{lem:Lem1}
Let $T>0$ be a fixed constant. Let $(x_0,y_0)\in {S^{n-1}}^{(2)}$, there is $\bar{\varepsilon}>0$ and $K>0$ depending continuously on $(x_0,y_0)$ such that for any $0<\varepsilon\le\bar{\varepsilon}$ and $x,x'\in B_{\varepsilon}(x_0)$ and $y,y'\in B_{\varepsilon}(y_0)$ we have
$$
\| tv(x,y) - tv(x',y') \| <Kt\varepsilon,
$$
and also 
$$
d_G(H(x,y,t), H(x',y',t)) < Kt\varepsilon
$$
for $0<t\le T$, where $K$ depends continuously on $(x_0, y_0)$.
\end{lemma}
\begin{proof}
We let $\bar{\varepsilon} = (1/3)d_S(x_0, y_0)$ which ensures that  $cl(B_{\bar{\varepsilon}}(x_0)) \cap cl( B_{\bar{\varepsilon}}(y_0) )= \emptyset$. Since  $v:(\partial_\infty\mathbb{H}^n)^{(2)}\rightarrow\mathfrak{g}$ is analytic, it is locally lipschitz, and thus lipschitz on compact subsets of $(\partial_\infty\mathbb{H}^n)^{(2)}$, in particular on  $cl({B_{\bar{\varepsilon}}(x_0)})\times  cl( B_{\bar{\varepsilon}}(y_0) )$. Let $K'$ be the infimum of all lipschitz constant for $v$ on this compact set. Indeed 
$$ K' = \sup \{ \frac{\| v(x,y) - v(x',y')\|}{d((x,y),(x',y') )} | (x,y),(x',y')\in cl({B_{\bar{\varepsilon}}(x_0)})\times  cl( B_{\bar{\varepsilon}}(y_0) ) \} $$
which depends continuously on $(x_0,y_0)$.

Let $0<\varepsilon<\bar{\varepsilon}$, for all $(x,y), (x',y')\in B_\varepsilon(x_0)\times B_\varepsilon(y_0)$,
$$
\|v(x,y) - v(x',y')\| \le K' d((x,y),(x',y')) < K' \varepsilon\sqrt{2}
$$
We let $K_1=K'\sqrt{2}$ and scale both sides by $t>0$ to get the first inequality.

The inverse exponential map from a neighborhood of $I\in G$ to $\mathfrak{g}$ is analytic and bijective, so it is lipschitz on compact subsets, in particular on the set $C_{x_0,y_0, T} = \{H(x,y,t) | (x,y)\in cl(B_{\bar{\varepsilon}}(x_0)) \times cl( B_{\bar{\varepsilon}}(y_0)), t\in[-T,T] \}$. Again we let $K''$ be the infimum of all lipschitz constants for the inverse exponential map on  $C_{x_0,y_0, T}$ which depends continuously on $(x_0, y_0)$ once $T$ is fixed. Then for all $(x,y), (x',y')\in B_\varepsilon(x_0)\times B_\varepsilon(y_0)$ and $0<t\le T$, 
$$
d_G(H(x,y,t), H(x',y',t)) < K'' \|tv(x,y) - tv(x',y')\|<K'' K_1 t\varepsilon.
$$
We let $K = K'' K_1$.
\end{proof}

\begin{notation}
Let $\mathcal{A}$ be the geodesic path from $\tilde{x}_0$ to $\tilde{x}_0 A$, and let $\lambda(\mathcal{A})$ be the $\lambda$-measure of this path. For our purpose, the constant $T$ in the above lemma is $\lambda(\mathcal{A}) +1$.
\end{notation}


Following the discussion in section \ref{subsec:MLS}, $A$ and $\lambda$ and a choice of base point $\tilde{x}_0\in\widetilde{S}$ determines a compact set $I_1\times I_2 \subset {\mathbb{S}^1_\infty}^{(2)}$, and Anosov representation $\rho$ determines the limit set homeomorphism $L$.

\begin{lemma}
\label{lem:compact}
For a compact set $D\subset G$, there exists a constant $C>0$ such that for any $\beta\in D$, $(x,y),(x',y')\in L(I_1)\times L(I_2)$, and weights $|t|, |t'| <\lambda(\mathcal{A}) +1$ we have
$$
d_G(H(x,y,t) \beta, H(x',y',t')\beta) < C d_G(H(x,y,t) , H(x',y',t')).
$$
\end{lemma}
\begin{proof}
We have that the map $(x,y,t,\beta) \mapsto H(x,y,t)\beta$ is analytic with respect to any reasonable coordinates. So the result follows from locally lipschitz argument and compactness of the domain of $(x,y,t,\beta)$. Note that $C$ depends on $D$, $L(I_1)\times L(I_2)$ and $\lambda(\mathcal{A})$.
\end{proof}

\begin{lemma}
\label{lem:LemDelta}
There is a constant $C>0$ depending on $L(I_1)\times L(I_2)$ and $\lambda(\mathcal{A})$ such that if we have $(x_1,y_1),... , (x_k,y_r)\in L(I_1)\times L(I_2)$ and real positive weights $t_1, t'_1,... , t_r, t'_r$ such that $|t_i - t'_i| < \delta$, and $\sum t_i, \sum t'_i < \lambda(\mathcal{A}) +1$. Then
$$
d_G(\prod_{i=1}^r H(x_i,y_i,t_i) , \prod_{i=1}^r H(x_i,y_i,t'_i) ) < C r \delta.
$$
\end{lemma}
\begin{proof} First note that by left-invariance, $d_G(H(x_i,y_i,t_i), H(x_i,y_i,t'_i)) = d_G(I, H(x_i,y_i, t'_i - t_i)$. By compactness and locally lipschitz argument, there exists a constant $C'>0$ depending on $L(I_1)\times L(I_2)$ and $\lambda(\mathcal{A})$ such that 
$ d_G(I, H(x,y,t)) < C' |t| $ for all $(x,y)\in L(I_1)\times L(I_2)$ and $|t| < \lambda(\mathcal{A})$. So 
$$
d_G(H(x_i,y_i,t_i), H(x_i,y_i,t'_i)) < C'\delta
$$ for all $i$.

Let constant $R = \max\{\frac{d_G(I,H(x,y,t))}{|t|} | (x,y)\in L(I_1)\times L(I_2) \textnormal{ and } |t|\le\lambda(\mathcal{A})+1 \}$. Let $D$ be the compact subset of distance at most $R (\lambda(\mathcal{A}) +1)$ from $I$. We have $\prod_{i=1}^k H(x_i,y_i,t_i)$ and $\prod_{i=1}^k H(x_i,y_i,t'_i)$ are in $D$. Using lemma \ref{lem:compact} we get for any $\beta\in D$, 
$$
d_G(H(x_i,y_i,t_i) \beta, H(x_i,y_i, t'_i)\beta)< C'' d_G(H(x_i,y_i,t_i), H(x_i,y_i,t'_i)) < C'' C' \delta.
$$
Therefore by triangle inequality and replacing terms of the product one by one from the right, we have 
$$
d_G(\prod_{i=1}^r H(x_i,y_i,t_i) , \prod_{i=1}^r H(x_i,y_i,t'_i) ) < r C'' C' \delta = C r \delta
$$
where $C$ depends only on $L(I_1)\times L(I_2)$ and $\lambda(\mathcal{A})$.
\end{proof}

\begin{theorem}
\label{Theorem1}
Given $\rho:\pi\rightarrow SO(n,1)$ Anosov, $A\in\pi_1(S)$, $\lambda\in\mathcal{ML}(S)$, for all $\epsilon>0$ there exists a neighborhood $U\subset \mathcal{ML}(S)$ of $\lambda$ such that for any two weighted simple closed curves $l_1, l_2\in U$ we have 
$$d_G(\mathcal{E}_{l_1}(\rho)(A), \mathcal{E}_{l_2}(\rho)(A) ) < \epsilon.$$
\end{theorem}
\begin{proof}
Recall our notation, here $\mathcal{E}_{l_i}(\rho)$ is the representation obtained from $\rho$ by algebraic deformation along the weighted simple closed curve $l_i$.
\begin{notation}For any $\mu\in\mathcal{ML}(S)$, we abbreviate $\mu^* = L_*\mu|_A$ the induced measure on $L(I_1)\times L(I_2)$.
\end{notation}
From lemma \ref{lem:Lem1} we have a continuous function $\bar{\varepsilon}:L(I_1)\times L(I_2)\rightarrow \mathbb{R}^+$, and on this compact set $\bar{\varepsilon}$ reaches a non-zero minimum value $\bar{\varepsilon}_\textnormal{min} = \min\{\bar{\varepsilon}(x,y) | (x,y)\in L(I_1)\times L(I_2)\} > 0.$
The $K$ in lemma \ref{lem:Lem1} depends continuously on the centers of the balls considered. So we let $K_{max}$ be a constant that works for all $(x,y)$ in the compact set $L(I_1)\times L(I_2)$. Following the proof of lemma \ref{lem:LemDelta} we let $D$ be the compact subset that contains all transformation of the form $\prod_{i=1}^k H(x_i,y_i,t_i)$ as long as $(x_i,y_i)\in L(I_1)\times L(I_2)$ and $\sum t_i \le \lambda(\mathcal{A})$. Let $C_0$ be the constant of lemma \ref{lem:compact} that works for $D$.

We choose $\varepsilon_0 <\bar{\varepsilon}_\textnormal{min}$, and also such that
 $$\varepsilon_0<\frac{\epsilon}{4C_0 K_{max}(\lambda(\mathcal{A}) +1)}$$ 
The reason for choosing as such will be apparent by the end.

We have for $i=1,2$, the compact set $L(I_i)$ can be partitioned into $m$ disjoint intervals $L(I_i)_1,...,L(I_i)_m$ such that each interval is of diameter less than $\varepsilon_0$, in particular each interval lie inside an open ball of radius $\varepsilon_0$ center at some point in $L(I_i)$.  We may use half-open-half-closed (topologically $(0,1]$) intervals to ensure they are disjoint. So $L(I_1)\times L(I_2)$ is partitioned into $m^2$ squares which we name $S_{i,j} = L(I_1)_i \times L(I_2)_j$ for $1\le i,j \le m$.
\begin{notation} For a square $S\in\{S_{i,j}\}$ we let $(x_S,y_S)$ be its center. That is, $(x_S, y_S)\in S$ such that $S\subset B_{\varepsilon_0}(x_S) \times B_{\varepsilon_0}(y_S) $
\end{notation}

Since only a finite number of geodesic leaf of $\lambda$ in $I_1\times I_2$ can have positive measure, we can assume each $S_{i,j}$ to be a continuity set with respect to $\lambda^*$, that is, their boundaries have $\lambda^*$-measure $0.$

{\bf{Choose}} $U\subset\mathcal{ML}(S)$ to be a neighborhood of $\lambda$ containing all measured laminations $\mu$ such that $|\lambda(\mathcal{A}) - \mu(\mathcal{A})|<1$ and $\lambda^*$ measure and $\mu^*$ measure are $\delta/2$ close on every square $S_{i,j}$. We choose $\delta$ such that
$$
\delta < \frac{\epsilon}{6 (2m - 1)C}
$$
where the constant $C$ is from lemma \ref{lem:LemDelta}. Note that $m$ depends on $\varepsilon_0$ which in turn depends on $\epsilon, \rho, \lambda, A$.
 
Let $l_1, l_2$ be weighted simple closed curves in $U$ with weights $t_1, t_2\in\mathbb{R}^+$ respectively. We have $l_1^*$ and $l_2^*$ are $\delta$ close on $S_{i,j}$.  Since $l_1$ is a simple curve, its lifts are disjoint from each other. So, suppose $S_{i,j}$ contains the endpoints of a lift of $l_1$, then $S_{i',j'}$ cannot contain endpoints of any lift of $l_1$ if either $i<i', j>j'$ or $i>i', j<j'$. In other words if both $S_{i,j}, S_{i',j'}$ contain endpoints of lifts of $l_1$ then either $i\le i', j\le j'$ or $i\ge i', j\ge j'$. So there are at most $2m-1$ squares containing endpoints of lifts of $l_1$ and there is a complete ordering on this set of $2m-1$ squares. Same statements can be made about $l_2$ or any other geodesic lamination. Let $\mathfrak{S}_1$ be the ordered set of squares in $\{S_{i,j}\}$ which contain the endpoints of lifts of $l_1$. Similarly we define the set $\mathfrak{S}_2$ for $l_2$. The set $\mathfrak{S}_1\cap\mathfrak{S}_2$ also has a complete ordering compatible with both $\mathfrak{S}_1$ and $\mathfrak{S}_2$.
 
Note that for $S\in \mathfrak{S}_1 - (\mathfrak{S}_1\cap\mathfrak{S}_2)$,  $l^*_2(S) = 0$, so $l^*_1(S) <\delta$. Similarly for $S\in \mathfrak{S}_2 - (\mathfrak{S}_1\cap\mathfrak{S}_2)$. Then by lemma \ref{lem:LemDelta} we have the following:
$$
d_G\left(\prod_{S\in\mathfrak{S}_1 }H(x_S,y_S,l^*_1(S)), \prod_{S\in\mathfrak{S}_1 \cap \mathfrak{S}_2}H(x_S,y_S,l^*_1(S)) \right)<(2m-1)C\delta,
$$
 $$
d_G\left(\prod_{S\in\mathfrak{S}_1 \cap \mathfrak{S}_2}H(x_S,y_S,l^*_1(S)), \prod_{S\in\mathfrak{S}_1\cap\mathfrak{S}_2 }H(x_S,y_S,l^*_2(S)) \right)<(2m-1)C\delta,
$$
$$
d_G\left(\prod_{S\in\mathfrak{S}_1 \cap \mathfrak{S}_2}H(x_S,y_S,l^*_2(S)), \prod_{S\in\mathfrak{S}_2 }H(x_S,y_S,l^*_2(S)) \right)<(2m-1)C\delta.
$$
Here the order in the products are determined by the ordering of $\mathfrak{S}_1, \mathfrak{S}_2$ and $\mathfrak{S}_1\cap\mathfrak{S}_2$. Therefore we have
\begin{equation}
\label{eq:estimate1}
d_G\left(\prod_{S\in\mathfrak{S}_1 }H(x_S,y_S,l^*_1(S)), \prod_{S\in\mathfrak{S}_2 }H(x_S,y_S,l^*_2(S)) \right) <3(2m-1)C\delta <\frac{\epsilon}{2}
\end{equation}
By the choice of $\delta$.
 
 Now we will show that $\rho(A)^{-1}\mathcal{E}_{l_1}(\rho)(A)$ is close $\prod_{S\in\mathfrak{S}_1 }H(x_S,y_S,l^*_1(S))$ and similarly for $l_2$.Suppose $(x^1_1,y^1_1), (x^1_2, y^1_2), ..., (x^1_{k_1}, y^1_{k_1})$ (in order) are the pairs of end points in $L(I_1)\times L(I_2)$ corresponding to the lifts of $l_1$ which intersect the geodesic path from $\tilde{x}_0$ to $\tilde{x}_0A$, and let $t_1\in\mathbb{R}^+$ be the weight of $l_1$. We have 
\begin{equation} 
\label{eq:decomp1}
\begin{split}
\rho(A)^{-1}\mathcal{E}_{l_1}(\rho)(A) & =  \prod_{i=1}^{k_1} H(x^1_i, y^1_i, t_1) = \prod_{S\in\mathfrak{S}_1}\left(\prod_{(x^1_i,y^1_i)\in S} H(x^1_i,y^1_i,t_1) \right) 
\end{split}
\end{equation}
On the other hand,
\begin{equation}
\label{eq:decomp2}
\prod_{S\in\mathfrak{S}_1 }H(x_S,y_S,l^*_1(S)) = \prod_{S\in\mathfrak{S}_1 }\left (\prod_{(x^1_i,y^1_i)\in S}H(x_S,y_S, t_1)\right)
\end{equation}
Recall the we have the compact set $D\subset G$ and that 
$\prod_{S\in\mathfrak{S}_1\cup\mathfrak{S}_2}H(x_S,y_S,t_S) \in D$ for any $t_S$ such that $\sum_{S\in\mathfrak{S}_1}t_S <\lambda(\mathcal{A}) +1$. Applying lemma \ref{lem:compact} and \ref{lem:Lem1} we have for any $\beta\in D$ and $(x^1_i, y^1_i)\in S  \in \mathfrak{S}_1$,
\begin{equation*}
\begin{split}
d_G(H(x_S,y_S,t_1) \beta, H(x^1_i,y^1_i,t_1)\beta)&  < C_0 d_G(H(x_S,y_S,t_1), H(x^1_i,y^1_i,t_1)) \\
& < C_0 K t_1 \varepsilon_0.
\end{split}
\end{equation*}
We can now estimate distance between (\ref{eq:decomp1}) and (\ref{eq:decomp2}) by replacing the all the terms one by one from the right, each time adding $C_0 K t_1\varepsilon_0$ to the distance. So 
\begin{equation}
\label{eq:estimate2}
\begin{split}
d_G\left(\rho(A)^{-1}\mathcal{E}_{l_1}(\rho)(A),\prod_{S\in\mathfrak{S}_1 }H(x_S,y_S,l^*_1(S)) \right)<\sum_{i=1}^{k_1} C_0 K t_1\varepsilon_0 = C_0 Kl_1^*(\mathcal{A})\varepsilon_0 \\
< C_0K_{max}(\lambda(\mathcal{A})+1)\varepsilon_0 <\epsilon/4
\end{split}
\end{equation}
The analogous inequality holds for $l_2$ as well. Therefore from (\ref{eq:estimate1}) and (\ref{eq:estimate2}) we have
$$
d_G(\mathcal{E}_{l_1}(\rho)(A), \mathcal{E}_{l_2}(\rho)(A)) =  d_G\left(\rho(A)^{-1}\mathcal{E}_{l_1}(\rho)(A), \rho(A)^{-1}\mathcal{E}_{l_2}(\rho)(A)\right) < \epsilon
$$
\end{proof}

Following \cite{KerckhoffNielsen} we choose a set of generators $A_1,...,A_{2g}$ for $\pi_1(S)$ and say that two representations $\rho_1, \rho_2$ are $\epsilon$ close if $d_G(\rho_1(A_i),\rho_2(A_i)) <\epsilon$ for each $A_i$. The induced topology on the space of representations is independent of the choice of generating set. We have the following corollary by applying Theorem \ref{Theorem1} to each generator and taking the finite intersection of open neighborhoods.

\begin{corollary}
\label{cor:convergence}
For any $\epsilon>0$, $\lambda\in\mathcal{ML}(S)$ an Anosov surface group representation $\rho$ into $SO(n,1)$, there is a neighborhood $U\ni \lambda$ such that for any two weighted simple closed curves $l_1, l_2\in U$, the corresponding representations $\mathcal{E}_{l_1}(\rho)$ and $\mathcal{E}_{l_2}(\rho)$ are $\epsilon$ close.
\end{corollary}

This implies that for any sequence of weighted simple closed curves $(l_i)$ converging to $\lambda\in\mathcal{ML}(S)$, the corresponding sequence of representations $\mathcal{E}_{l_i}(\rho)$ is Cauchy and thus converges. Moreover, the limit of this sequence does not depend on the choice of sequence converging to $\lambda$, so we can define $\mathcal{E}_\lambda(\rho) := {\lim}_{i\rightarrow\infty} \mathcal{E}_{l_i}(\rho)$ which is also a representation of $\pi_1(S)$. Thus we have a continuous map from $\mathcal{ML}(S)$ to the connected components of representations near $\rho$.

\begin{thebibliography}{1}

\bibitem{AramayonaLeininger}J. Aramayona and C. J. Leininger,
  \emph{Hyperbolic structures on surfaces and geodesic currents},
  Lecture notes (2014).

\bibitem{BonahonCurrents} F. Bonahon,
  \emph{The geometry of Teichm\"{u}ller space via geodesic currents},
  Invent. Math. 92 (1988), 139--162.
  
\bibitem{BonahonShearBend} F. Bonahon,
  \emph{Shearing hyperbolic surfaces, bending pleated surfaces and Thurston’s symplectic form},
  Ann. Fac. Sci. Toulouse Math. 6 (1996), 5(2): 233--297

\bibitem{BowditchGeomFinite} B. H. Bowditch,
  \emph{Geometrical finiteness with variable negative curvature},
  Duke Math. J. 77 (1995), no. 1, 229--274.

\bibitem{GoldmanSymplecticNature} W. M. Goldman,
   \emph{The symplectic nature of fundamental groups of surfaces},
   Adv. in Math. 54 (1984), no. 2, 200--225.
   
\bibitem{GoldmanConvexProj} W. M. Goldman,
  \emph{Convex projective structures on compact surfaces},
  J. Differential Geom. 31 (1990), 791--845.

\bibitem{GuichardWienhardAnosov} O. Guichard and A. Wienhard,
  \emph{Anosov representations: domains of discontinuity and applications},
  Invent. math. (2012) 190: 357. doi:10.1007/s00222-012-0382-7

\bibitem{KTDefo} Y. Kamishima and S. P. Tan,
  \emph{Deformation spaces on geometric structures},
  Aspects of low-dimensional manifolds, Kinokuniya, Tokyo (1992), 263--299.

\bibitem{KapovichHigherKlein} M. Kapovich,
  \emph{Kleinian groups in higher dimensions},
  Progress in Math. Vol. 265, 485--562.

\bibitem{KLP2016notes} M. Kapovich and B. Leeb and J. Porti,
  \emph{Lectures on Anosov representations I: Dynamical and geometric characterizations},
  preprint (2016).
         
\bibitem{KerckhoffNielsen} S. P. Kerckhoff,
  \emph{The Nielsen realization problem},
  The Annals of Math. 117 (1983), no. 2, 235--265.
  
\bibitem{KComplexLength} C. Kourouniotis,
  \emph{Complex length coordinates for quasifuchsian groups},
  Mathematika, 41 (1994), 173--188.

\bibitem{JohnsonMillson} D. Johnson and J. J. Millson,
  \emph{Deformation spaces associated to compact hyperbolic manifolds},
   Discrete Groups in Geometry and Analysis, Vol. 67 Progress in Math., 48--106.
   
\bibitem{McMullenComplexEq} C. McMullen,
  \emph{Complex earthquakes and Teichmuller theory},
  J. Amer. Math. Soc., 11, (1998), no. 2, 283--320.
  
\bibitem{TanComplexFN} S. P. Tan,
  \emph{Complex Fenchel-Nielsen coordinates for quasi-Fuchsian structures},
  Internat. J. Math 5 (1994), no. 2, 239--251.
  
\bibitem{Tan4Hexagon} S. P. Tan and Y. L. Wong and Y. Zhang,
  \emph{Generalized Delambre-Gauss formulas for oriented, augmented, right-angled hexagons in hyperbolic 4-space},
  Adv. Math. 230, (2012), 927--956.
  
\bibitem{ThurstonNotes} W. P. Thurston,
  \emph{The geometry and topology of 3-manifolds},
  Princeton University Lecture Notes, 1982, online at http://www.msri.org/publications/books/gt3m

  




   


\end{thebibliography}
\end{document}